\documentclass[11pt,reqno]{amsart}

\makeatletter
\usepackage{amssymb}
\usepackage{amssymb,amsmath,color,comment}
\usepackage{amsbsy}
\usepackage{amsfonts}

\usepackage{enumitem}

\def\marginpar#1{\ignorespaces}

\newtheorem{thm}{Theorem}
\newtheorem{prop}[thm]{Proposition}
\newtheorem{lem}[thm]{Lemma}

\theoremstyle{definition}

\newtheorem{rmq}[thm]{Remark}

\def\eqlabel#1{\def\@currentlabel{#1}}

\def\formula#1{\def\@tempa{#1}\let\@tempb\theequation\def\theequation{%
\hbox{#1}}\def\@currentlabel{(\theequation)}$$}
\def\endformula{\leqno\hbox{(\@tempa)}$$\@ignoretrue\let\theequation\@tempb}

\def\given{\hskip5\p@\relax\vrule\@width.4\p@\hskip5\p@\relax}

\newcommand{\open}[1]{%
\par\normalfont\topsep6\p@\@plus6\p@\trivlist\item[\hskip\labelsep\itshape#1%
\@addpunct{.}]\ignorespaces}

\DeclareRobustCommand{\close}[1]{%
  \ifmmode 
  \else \leavevmode\unskip\penalty9999 \hbox{}\nobreak\hfill
  \fi
  \quad\hbox{$#1$}}

\newlength{\toskip}

\makeatother

\renewcommand{\theequation}{\thesection.\arabic{equation}}

\numberwithin{equation}{section}

\renewcommand\H{\mathcal{H}}

\newcommand{\RR}{{\mathbb R}}

\newcommand{\dd}{\text{d}}

\newcommand{\ent}{\mathop{\rm Ent\,}\nolimits}

\newcommand{\Rom}[1]{\uppercase\expandafter{\romannumeral #1\relax}}

\newtheorem*{assumption*}{\assumptionnumber}
\providecommand{\assumptionnumber}{}
\makeatletter
\newenvironment{assumption}[2]
 {%
  \renewcommand{\assumptionnumber}{\textbf{#1} $\textbf{#2}$}%
  \begin{assumption*}%
  \protected@edef\@currentlabel{ $\textbf{#2}$}%
 }
 {%
  \end{assumption*}
 }
\makeatother

\begin{document}

\title{The kinetic Fokker-Planck equation with mean field interaction}


\author[Arnaud Guillin]{\textbf{\quad {Arnaud} Guillin $^{\spadesuit}$ \, \, }}
\address{{\bf {Arnaud} GUILLIN}\\ Laboratoire de Math\'ematiques Blaise Pascal, CNRS UMR 6620, Universit\'e Clermont Auvergne,
Campus des C\'ezeaux 3, place Vasarely , F-63178 Aubi\`ere.} \email{arnaud.guillin@uca.fr}

\author[W. Liu]{\textbf{\quad {Wei} Liu $^{\clubsuit}$ \,  }}
\address{{\bf Wei LIU} School of Mathematics and Statistics, Wuhan University, Wuhan, Hubei 430072, P.R. China; Computational Science Hubei Key Laboratory, Wuhan University, Wuhan, Hubei 430072, P.R. China.} \email{wliu.math@whu.edu.cn}

\author[Liming Wu]{\textbf{\quad {Liming} Wu $^{\spadesuit}$ \, \, }}
\address{{\bf {Liming} WU}\\ Laboratoire de Math\'ematiques Blaise Pascal, CNRS UMR 6620, Universit\'e Clermont Auvergne,
Campus des C\'ezeaux 3, place Vasarely , F-63178 Aubi\`ere.} \email{li-ming.wu@math.univ-bpclermont.fr}

\author[Chaoen Zhang]{\textbf{\quad {Chaoen} Zhang $^{\spadesuit}$ \, \, }}
\address{{\bf {Chaoen} ZHANG}\\ Laboratoire de Math\'ematiques Blaise Pascal, CNRS UMR 6620, Universit\'e Clermont Auvergne,
Campus des C\'ezeaux 3, place Vasarely , F-63178 Aubi\`ere.} \email{chaoen.zhang@uca.fr}

\maketitle

 \begin{center}

\textsc{$^{\spadesuit}$  Universit\'e Clermont Auvergne}
\smallskip

\textsc{$^{\diamondsuit}$ Wuhan University}
\smallskip

\end{center}

\begin{abstract}
We study the long time behaviour of the kinetic Fokker-Planck equation with mean field interaction, whose limit is often called Vlasov-Fkker-Planck equation. We prove a uniform (in the number of particles) exponential convergence to equilibrium for the solutions in the weighted Sobolev space $H^1(\mu)$ with a rate of convergence which is explicitly computable and independent of the number of particles. The originality of the proof relies on functional inequalities and hypocoercivity with Lyapunov type conditions, usually not suitable to provide adimensional results.
\end{abstract}

\bigskip

\bigskip

\section{Introduction}

In this paper we are interested in the system of $N$ particles moving in $\RR^d$ with mean field interaction
\begin{equation}\label{kFPsde2}
\left\{
\begin{aligned}
\dd x^i_t &= v^i_t\dd t\\
\dd v^i_t &= \sqrt{2}\dd B^i_t -v^i_t\dd t -\nabla U(x^i_t)- \frac{1}{N}\sum_{1\leq j\leq N}\nabla W(x^i_t-x^j_t)\dd t
\end{aligned}
\right.
\end{equation}
where $x_t^i, v_t^i$ are respectively the position and the velocity of the $i$-th particle, and $(B_t^i)_{t\geq 0} (1\leq i\leq N)$ are independent standard Brownian motions on $\RR^d$, $U:\RR^d\rightarrow \RR$ is the confinement potential, and $W:\RR^d\rightarrow \RR$ is the interaction potential. Equivalently, denote $(x_t,v_t)=((x_t^1,x_t^2,\cdots,x_t^N),(v_t^1,v_t^2,\cdots,v_t^N))$, the particle system can be rewritten in a more compact form
\begin{equation}\label{kFPsde1}
\left\{
\begin{aligned}
\dd x_t &= v_t\dd t\\
\dd v_t &= \sqrt{2}\dd B_t -v_t\dd t -\nabla V(x_t)\dd t
\end{aligned}
\right.
\end{equation}
where $B_t=(B_t^1, B^2_t,\cdots,B_t^N)$ and the function $V$ is the whole potential with mean field interaction given by
\begin{equation}\label{VNd}
V(x_1,x_2,\cdots,x_N)= \sum\limits_{1\leq i\leq N}U(x_i) + \frac{1}{2N}\sum\limits_{1\leq i,j\leq N}W(x_i-x_j).
\end{equation}

This damping stochastic Newton equation, though non-elliptic, is hypoelliptic. It has a unique invariant probability measure $\mu(\dd x, \dd v)$ on $\RR^{Nd}\times \RR^{Nd}$ given by
\[
\mu(\dd x,\dd v) = \frac{1}{Z}e^{-V(x)}\cdot(2\pi)^{-\frac{Nd}{2}}e^{-\frac{|v|^2}{2}}\dd x\dd v
\]
where $x=(x_1,x_2,\cdots,x_N), v= (v_1,v_2,\cdots,v_N)$ with $x_i,v_i \in \RR^d$ for $1\leq i\leq N$, and $Z$ is the normalization constant (called often the partition function). Denote
\[
\dd m(x)=\frac{1}{Z}e^{-V(x)}\dd x,\quad \dd\gamma(v)=(2\pi)^{-\frac{Nd}{2}}e^{-\frac{|v|^2}{2}}\dd v
\]
and so $\mu(\dd x,\dd v)=\dd m(x)\dd\gamma(v)$.

The density function $h_t(x,v)= \dd\mu_t(x,v)/\dd\mu(x,v)$ of the law $\mu_t$ of the diffusion process $(x_t,v_t)$ with respect to the equilibrium measure $\mu$ satisfies the kinetic Fokker-Planck equation on $\RR^{Nd}\times \RR^{Nd}$
\begin{equation}\label{kFP1}
  \frac{\partial h}{\partial t} + v\cdot \nabla_x h - \nabla_x V(x)\cdot\nabla_vh = \Delta_vh - v\cdot\nabla_vh
\end{equation}
subject to the initial condition $h_0(x,v)= \dd\mu_0(x,v)/\dd\mu(x,v)$. Here $a\cdot b$ denotes the Euclidean inner product of two vectors $a$ and $b$, $\nabla_x$ stands for the gradient with respect to the position variable $x\in \RR^{Nd}$, whereas $\nabla_v$ and $\Delta_v$  stand for the gradient and the Laplacian with respect to the velocity variable $v\in \RR^{Nd}$, respectively. And we shall adopt the notation $\nabla^2$ for the Hessian operator, and $\nabla_{xv}^2=(\partial^2/\partial x_k\partial v_l)_{1\leq k,l\leq Nd}$ for the mixed Hessian operator.

We denote by $L^2(\mu)$ the weighted $L^2$ space with respect to the reference measure $\mu$ for which $||\cdot||$ is the $L^2(\mu)$-norm and $\langle\cdot,\cdot\rangle$ is the associated inner product. Denote by $H^1(\mu)$ the weighted $L^2$-Sobolev space of order $1$ with respect to $\mu$, and the norm  $||\cdot||_{H^1(\mu)}$ is given by
\begin{equation}
  ||h||_{H^1(\mu)}^2 := \int h^2\dd\mu + \int \left(|\nabla_x h|^2(x,v) + |\nabla_vh|^2(x,v)\right)\dd\mu(x,v).
\end{equation}

When the probability measure $m$ satisfies a Poincar\'e inequality, and when $\nabla^2V$ satisfies some "boundedness" condition (see the condition \eqref{HypV2Eq} below), C. Villani \cite{Villani} established the exponential convergence of $h_t$ in $H^1(\mu)$. This is the starting of the term "hypocoercivity" method, which was before  initiated by \cite{DV,HN,H1}. An other approach was initiated by Dolbeault-Mouhot-Schmeiser \cite{DMS1,DMS2} with the advantage of  not needing a priori regularity results. Their $H^1$-convergence holds under the same assumptions. Note that is has triggered quite a lot of results for kinetic equations \cite{DR,MM,Cal,Cao1,Cao2,E}. However Both Villani's and DMS's approach on the exponential convergence rate depends highly on the number $N$ of particles. To complete this review on the speed to equilibrium for the Langevin equation, let us mention that a probabilistic approach based on coupling \cite{EGZ} or Lyapunov conditions \cite{Talay,Wu} was also developed but, as is often usual for Meyn-Tweedie's approach relying on Lyapunov conditions, the rate also depends (even more dramatically) on the dimension. Note however that, under very strong convexity assumptions, Bolley$\&$-al \cite{BGM10} obtained a uniform decay in Wasserstein distance for the mean field Langevin equation by a coupling approach. Very recently, an interesting work by Monmarch\'e \cite{M17} established an entropic decay, using Villani's hypocoercivity, but still under strong convexity assumptions, and Baudoin$\&$-al \cite{BGH} mixed Bakry's $\Gamma_2$ approach with hypocoercivity to obtain $H^1$ exponential decay even in a non regular case, i.e. Lennard-Jones potential, but with a rate still depending on the dimension. Note also that for a non mean-field case but oscillators Menegaki \cite{M19} obtained a dimension dependent convergence to equilibrium. The objective of this work is to establish, and it seems to be the first result under non convexity assumptions on the potential, some exponential convergence in $H^1(\mu)$, uniform in the number $N$ of particles. The originality of our approach is that we will combine Villani's hypocoercivity with recent uniform functional inequality and Lyapunov conditions (usually not suitable to provide adimensional results).\\

As an other motivation to get uniform in the number of particles result, the linear diffusion process $(x_t,v_t)_{t\geq 0}$ in $\RR^{Nd}\times\RR^{Nd}$ is the mean field approximation of the self-interacting diffusion process $(\bar{x}_t,\bar{v}_t)_{t\geq 0}$ in $\RR^d\times \RR^d$ which evolves according to
\begin{equation}\label{kFPsde3}
\left\{
\begin{aligned}
\dd \bar{x}_t &= \bar{v}_t\dd t\\
\dd \bar{v}_t &= \sqrt{2}\dd \bar{B}_t -\bar{v}_t\dd t -\left[\nabla U(\bar{x}_t)+ \int \nabla W(\bar{x}_t-y)u_t(\dd y)\right]\dd t
\end{aligned}
\right.
\end{equation}
where $u_t(\dd y)$ is the law of $\bar{x}_t$, and $\bar{B}$ is a standard Brownian motion on $\RR^d$. Its equivalent analytic version is: the density function $g_t=g(t,\bar{x},\bar{v})$ of the law of $(\bar{x}_t,\bar{v}_t)_{t\geq 0}$ with respect to the Lebesgue measure $\dd \bar{x}\dd \bar{v}$ satisfies the following self-consistent Vlasov-Fokker-Planck equation on $\RR^d\times\RR^d$
\begin{equation}\label{VFP}
 \frac{\partial g}{\partial t} + \bar{v}\cdot \nabla_{\bar{x}} g - (\nabla U(\bar{x})+ \nabla W* \pi g)\cdot\nabla_{\bar{v}}g = \Delta_{\bar{v}}g +\nabla_{\bar{v}}\cdot(\bar{v}g)
\end{equation}
subject to the initial condition that $g_0(\bar{x},\bar{v})$ is given by the law of $(x^1_0,v^1_0)$, where
\[
\pi g(\bar{x}) = \int_{\RR^d}g(t,\bar{x},w)\dd w
\]
is the macroscopic density in the space of positions $\bar{x}\in \RR^d$. This kinetic equation describes the evolution of clouds of charged particles, and it is significant in plasma physics (see Villani \cite{Villani} and references therein). Only very few results on the long time behavior of this nonlinear equation is known, see however \cite{Villani} in the compact valued case or Bolley {\em et al.}\cite{BGM10} in the strictly convex case (see also \cite{M17}). Our results are a first step towards such a long time behavior but the $H^1$ convergence does not behave well with respect to the dimension. We thus plan for a future work to consider entropic convergence and propagation of chaos for the mean field Langevin equation.\\\smallskip

Let us finish this introduction with the plan of our paper. The next Section presents the main assumptions and the main results, i.e. a uniform exponential convergence to equilibrium in $H^1$ under non convex assumptions. It also presents a crucial tool: Villani's hypocoercivity theorem. Its details will be given in Section 3. Section 4 contains useful lemmas in the case where the interaction potential has a bounded hessian. The next sections present the proofs of our main results: Theorem 3 in Section 5 and Theorem 4 in Section 6. The final Section presents a discussion on an improvement on the rate of convergence.


\section{Main results}
\subsection{Framework}
As in the introduction, $\dd m(x)= \frac{1}{Z}e^{-V(x)}\dd x$ is the probability measure on the position space $\RR^{Nd}$ and will be referred as the mean field measure later. Let $\dd \gamma(v)$ be the standard gaussian measure on the velocity space $\RR^{Nd}$, so $\dd\mu(x,v)= \dd m(x)\dd\gamma(v)$. 

Now we introduce our assumptions.
\begin{assumption}{}{(A1)}\label{Hyp0}{\it
  The functions $U$ and $W$ are twice continuously differentiable on $\RR^d$, $W$ is even (that is, $W(x)=W(-x)$ for all $x$), and
  \[
  Z=\int_{\RR^{Nd}} e^{-V(x)}\dd x < \infty, \quad \forall N\geq 2.
  \]
  i.e. $m$ is always assumed to be a probability measure.}
\end{assumption}

\begin{assumption}{}{(A2)}\label{Hyp1}{\it $\nabla^2W$ is bounded, i.e. there exists a positive constant $K$ such that
\[
-K I_{d}\leq  \nabla^2W\leq K I_{d}
\]
as quadratic forms on $\RR^d$, where $I_{d}$ is the identity matrix of size $d$.}
\end{assumption}

This assumption, which of course relaxes convexity, has been also considered in the propagation of chaos problem as well as the convergence of the (non kinetic) McKean-Vlasov equation in \cite{EGZ1,DEGZ}.

\begin{assumption}{}{UPI}\label{HypV1} {\it The measure $\dd m(x)= \frac1Z e^{-V(x)}\dd x$ satisfies a uniform Poincar\'e inequality i.e. there exists a positive real number $\kappa>0$ such that for any $N\geq 2$, and any compact-supported smooth function $h$ on $\RR^{Nd}$, it holds
     \begin{equation}\label{HypV1eq}
      \kappa\int \left(h-\int h\dd m\right)^2 \dd m\leq \int |\nabla_x h|^2 \dd m.
      \end{equation}}
\end{assumption}

The most easy-to-check criterion might be the Bakry-Emery curvature-dimension condition $CD(\kappa,\infty)$ (see for instance \cite{BGL}). It says that both Poincar\'e inequality and logarithmic Sobolev inequality (see \eqref{LSI} below) hold true for $\dd m(x)= \frac1Z e^{-V(x)}\dd x$ as soon as
\[
\nabla^2V(x)\geq \kappa I_{Nd}
\]
in the sense of quadratic forms on $\RR^{Nd}$. It can be verified if there exist constants $\kappa_1,\kappa_2$ such that
\begin{equation}
\nabla^2U\geq \kappa_1 I_{d}>0, \nabla^2W\geq \kappa_2I_{d}
\end{equation}
as quadratic forms on $\RR^{d}$, with $\kappa=\kappa_1-\kappa_2^{-}>0$ where $\kappa_2^{-}$ is the negative part of $\kappa_2$. Indeed, by Lemma \ref{LemHessW} below, the above inequalities imply that the contribution of the interaction potential $W$ in $\nabla^2V$ is bounded from below by $-\kappa_2^{-}I_{d}$, and the contribution of the confinement potential $U$ is bounded from below by $\kappa_1I_{d}$. Hence we have that $\nabla^2V\geq (\kappa_1-\kappa_2^{-})I_{Nd}$ as quadratic forms. It should be noted that $\kappa$ is then independent of the number $N$ of particles, i.e. we obtain a family of uniform functional inequalities for the mean field measure. Note that this strong convexity assumptions are the one employed in \cite{BGM10} for convergence in Wasserstein distance and by \cite{M17} for entropic convergence.

Other assumptions, more specified to the mean field measure $m$ for the uniform Poincar\'e inequalities and logarithmic Sobolev inequalities, can be found in another work \cite{GLWZ} of the authors. Indeed they proved these two functional inequalities with uniform (with respect to the number $N$ of particles) constants under various conditions on the confinement and interaction potentials, even when $U$ has two or more wells, and no convexity conditions on $W$. The methods used there depend on some dissipativity rate of the drift at distance $r>0$, defined by
\begin{equation}
  b_0(r)=\sup\limits_{x,y,z\in \RR^d:|x-y|=r}-\langle\frac{x-y}{|x-y|}, \nabla U(x)-\nabla U(y)+ \nabla W(x-z)-\nabla W(y-z)\rangle.
\end{equation}

\begin{thm}
Assume that the following Lipschitzian constant $c_{Lip,m}$ is finite
\begin{equation}
  c_{Lip,m}:= \frac{1}{4}\int_0^\infty \exp\left\{\frac14 \int_0^s b_0(u)\dd u\right\} s\dd s <\infty.
\end{equation}
Assume that there exists some constant $h> -1/c_{Lip,m}$ such that for any $(x_1,x_2,\cdots,x_N)\in \RR^{Nd}$,
\begin{equation}\label{UPIW}
\frac{1}{N}(-1_{i\neq j}\nabla^2W(x_i-x_j))_{1\leq i,j\leq N} \geq h I_{Nd}
\end{equation}
as quadratic forms. Then the mean field measure $m$ satisfies the following Poincar\'e inequality
\[
 (h+ 1/c_{Lip,m})\int \left(h-\int h\dd m\right)^2 \dd m\leq \int |\nabla_x h|^2 \dd m.
\]
for any function $h\in H^1(m)$.
\end{thm}

Recall that some nonnegative function $f\in L\log L(\mu)$, its entropy w.r.t. the probability measure $\mu$ is defined by
$$
\ent_\mu(f):= \int f\log f d\mu - \mu(f) \log \mu(f), \ \mu(f):=\int fd\mu.
$$

\begin{thm} Assume that
\begin{enumerate}
\item There exists a constant $\rho_{\rm LS,m}>0$ such that for all $i$ and $x^{\hat i}$, $m_i$, the conditional marginal distributions $m_i:=m_i(dx_i|x^{\hat i})$ of  $x_i\in\RR^d$ knowing $x^{\hat i}=(x_j)_{j\ne i}$, satisfies the log-Sobolev inequality :
    \begin{equation}\label{thm2-a1}
  \rho_{\rm LS,m}\ent_{m_i}(f^2)\le 2 \int |\nabla f|^2 {\rm d}m_i, \ f\in C^1_b(\RR^d).
\end{equation}
\item (a translation of Zegarlinski's condition)
\[
\gamma_0=c_{Lip,m}K<1.
\]
\end{enumerate}
then $m$ satisfies
\[
\rho_{\rm LS,m}(1-\gamma_0)^{2} \ent_{m} (f^2) \le 2 \int_{(\RR^d)^N} |\nabla f|^2 \dd m,\ f\in C^1_b((\RR^d)^N)
\]
i.e. the log-Sobolev constant of $m$ verifies
\[
\rho_{\rm LS}(m)\ge\rho_{\rm LS,m}(1-\gamma_0)^{2}.
\]
\end{thm}

We remark that the assumptions can be verified in various settings for which we refer to \cite{GLWZ}. For instance, the uniform logarithmic Sobolev inequalities for the conditional marginal measure can be verified by the Bakry-\'Emery $\Gamma_2$-criterion and the bounded perturbation theorem.

We will provide later explicit conditions on $V$ and $W$ to get such a result.

\subsection{Villani's hypocoercivity theorem}
We shall present Villani's hypocoercivity theorem for kinetic Fokker-Planck equation concerning the convergence to equilibrium (c.f. \cite{Villani} Theorem 35, Theorem 18). In the sequel we shall adopt the semigroup formulation. Set
\begin{equation}
-L:= \Delta_v - v\cdot \nabla_v - v\cdot \nabla_x  +\nabla V(x) \cdot \nabla_v,
\end{equation}
then the kinetic Fokker-Planck equation can be rewritten as
\[
\partial_t h + Lh=0.
\]
The associated semigroup will be denoted as $e^{-tL}$ and a solution could be represented by
\[
h(t,x,v)=e^{-tL}h(0,\cdot,\cdot).
\]
We shall use the notation $|S|_{HS}^2:=\sum\limits_{i,j}|S_{ij}^2 h|^2$ for the square of the Hilbert-Schimidt norm of the square matrix $S=(S_{ij})$. For instance, $|\nabla^2_{xv}h|_{HS}^2:=\sum\limits_{i,j}|\partial_{x_iv_j}^2 h|^2$. And for a square matrix $S$, $|S|_{op}$ stands for its operator norm.
\medskip

Villani's Hypocoercivity theorem in $H^1(\mu)$ (see \cite[Theorem 35]{Villani}) states,

\begin{thm}\label{Thm-VillaniH1}
Let $V$ be a $C^2$ function on $\RR^{Nd}$, satisfying the condition \ref{HypV1}. Suppose that there exists a positive real number $M$ such that
 \begin{equation}\label{HypV2Eq}
    \int |\nabla_x^2V(x)\cdot \nabla_v h|^2\dd\mu \leq M\left(\int |\nabla_vh|^2\dd\mu + \int |\nabla^2_{xv}h|_{HS}^2\dd\mu\right).
  \end{equation}
  for any $h \in H^2(\mu)$. Then there are constants $C_0>0$ and $\lambda>0$, explicitly computable, such that for all $h_0\in H^1(\mu)$
\begin{equation}\label{VillaniH1}
||e^{-tL}h_0 - \int h_0 \dd\mu||_{H^1(\mu)}\leq C_0e^{-\lambda t}||h_0||_{H^1(\mu)}.
\end{equation}
\end{thm}

The idea in Villani's proof of Theorem \ref{Thm-VillaniH1} is as follows: if one could find a Hilbert space such that the operator $L$ is coercive with respect to its norm, then one has exponential convergence for the semigroup $e^{-tL}$ under such a norm; If, in addition, this norm is equivalent to some usual norm (such as $H^1(\mu)$-norm), then one obtains exponential convergence under the usual norm as well.

We shall refer to the condition \eqref{HypV2Eq} as the boundedness condition \eqref{HypV2Eq} on $\nabla^2V$. In his statement of \cite[Theorem 35]{Villani}, this boundedness condition is verified by $|\nabla_x^2 V|\leq C(1+|\nabla V|)$ with a constant $M$ depending unfortunately on the dimension.

In the setting with mean field interaction, the constants $C_0$ and $\lambda$ given in \cite{Villani} depend on the number $N$ of particles, through the dependence of $M$ (in \eqref{HypV2Eq}) on $N$. In fact, by a careful analysis of the study in \cite{Villani}, we are led to the following observation: in \cite[Theorem 35, Lemma A.24]{Villani}, as $N\rightarrow \infty$, $\lambda$ decays faster than $N^{-2}$, while $C_0$ grows faster than $N^{3/2}$. We will give conditions under which we may bypass this dependence in the number of particles.

\subsection{Main results}

We have two different assumptions on the interaction potential ensuring an $H^1$ convergence to equilibrium. The first one is quite strong, namely that $W$ is a Lipschitzian function but we only assume a uniform Poincar\'e inequality (UPI).

\subsubsection{case {\bf UPI} and $|\nabla W|$ bounded}

\begin{thm}\label{Thm1}Assume \ref{Hyp0},\ref{Hyp1} and the condition \ref{HypV1}.
Suppose furthermore that $|\nabla W|\leq K'$ and the following Lyapunov condition holds
\begin{equation}\label{Thm1EqLy}
 |\nabla^2U|_{op} \leq K_1|\nabla U|+ K_2
\end{equation}
for some positive constants $K', K_1,K_2$. Then there exist explicitly computable constants $C_0$ and $\lambda$, independent of the number $N$ of the particles, such that
\begin{equation}
||e^{-tL}h_0 - \int h_0 \dd\mu||_{H^1(\mu)}\leq C_0e^{-\lambda t}||h_0||_{H^1(\mu)}
\end{equation}
for all $h_0\in H^1(\mu)$.
\end{thm}


\subsubsection{case Uniform Logarithmic Sobolev Inequality and {\bf (A2)}}$~$\\
In the next theorem, we shall release the boundedness assumption on $\nabla W$, but reinforce the condition \ref{HypV1} as
\begin{assumption}{}{ULSI}\label{Hyp3}{\it
  The mean field measure $m$ satisfies a uniform log-Sobolev inequality with a constant $C_{LS}>0$, i.e. for all $N\geq 2$ and for all smooth compactly-supported function $g$ on $\RR^{Nd}$, it holds
\begin{equation}\label{LSI}
\ent_{m}(g^2):=\int g^2\log g^2 \dd m -\int g^2\dd m \log\left(\int g^2\dd m\right)\leq 2C_{LS} \int |\nabla g|^2\dd m.
\end{equation}}
\end{assumption}
In \cite{GLWZ} practical conditions are given to ensure such a condition, see example below.

\begin{thm}\label{Thm2}Assume \ref{Hyp0},\ref{Hyp1} and the condition \ref{Hyp3}. Suppose furthermore that the Lyapunov condition \eqref{Thm1EqLy} holds for some positive constants $K_1$ and $K_2$. Then there exist explicitly computable constants $C_0$ and $\lambda$, independent of the number $N$ of the particles, such that
\begin{equation}
||e^{-tL}h_0 - \int h_0 \dd\mu||_{H^1(\mu)}\leq C_0e^{-\lambda t}||h_0||_{H^1(\mu)}
\end{equation}
for all $h_0\in H^1(\mu)$.
\end{thm}

We relax in this theorem the strong assumption concerning the boundedness of $|\nabla W|$ but we reinforce the functional inequality needed to ensure the adimensional result.

\subsection{Examples}

\subsubsection{UPI and Theorem 3}

Let assume the following convexity at infinity assumptions on $U$: there exists constants $c_U$, $c$ and $R\ge0$ such that
\begin{equation}\label{convinfi}
\langle \nabla U(x)-\nabla U(y),x-y\rangle\ge c_U|x-y|^2-c|x-y|1_{|x-y|\le R}.
\end{equation}
By following \cite[Cor. 5, Rem. 4]{GLWZ}, then assuming {\bf (A2)}, if we suppose moreover
$$(c_U-K)e^{-cR/4}-2K>0,$$
then {\bf UPI} holds. The Lyapunov condition \eqref{Thm1EqLy}, expressing that $U$ cannot grow too fast (more than exponentially) and the boundedness condition of $|\nabla W|$ are easy to verify.

\subsubsection{ULSI and Theorem 4}
For simplicity, we will suppose that $U$ is super convex at infinity, i.e. for any $\tilde K>0$ there exists $R>0$ such that
$$\nabla^2U\ge \tilde K \,I,\qquad \forall |x|\ge R.$$
Note that it implies \eqref{convinfi}. Suppose also
$$\frac{e^{cR/4}}{(c_U-K)}K<1$$
where $c$ and $C_U$ are described in \eqref{convinfi}, then a ULSI holds and once again the Lyapunov condition can be easily verified on examples.

\section{Villani's hypocoercivity theorem}
This section is devoted to Villani's hypocoercivity theorem. The following outline of the proof of \cite[Theorem 35]{Villani} further details the use of the condition \ref{HypV1} and the boundedness condition \eqref{HypV2Eq},
\begin{enumerate}
          \item Introduce an inner product $((\cdot,\cdot))$ in the form of
          \begin{equation}\label{NewNorm}
          ((h,h)) = ||h||^2 + a||\nabla_v h||^2 + 2b\langle \nabla_vh,\nabla_xh\rangle + c||\nabla_x h||^2
          \end{equation}
          where  the coefficients $a,b,c$ will be specified later such that 
          \begin{equation}\label{EqvNorm}
          c_1||h||_{H^1(\mu)}\leq ((h,h))^{1/2}\leq c_2||h||_{H^1(\mu)},\quad \forall h\in H^1(\mu)
          \end{equation}
          for some constants $c_1>0, c_2>0$.
          \item Prove a coercivity estimate for $L$ under the new inner product. Thanks to the boundedness condition \eqref{HypV2Eq}, one can choose appropriately the constants $a, b$ and $c$ such that
          \begin{equation}\label{EstmCoerL1}
          ((h,Lh))\geq \lambda_0(||\nabla_x h||^2+ ||\nabla_v h||^2), \quad \mbox{if} \int h\dd\mu=0 
          \end{equation}
          for some constant $\lambda_0>0 $ depending only on the constant $M$. By the tensorization property of Poincar\'e inequality, the condition \ref{HypV1} implies that
          \[
          ((h,h))\leq (2a+1)||\nabla_v h||^2+ (2c+\kappa^{-1})||\nabla_x h||^2
          \]
          for all function $h\in H^1(\mu)$ with $\int h\dd\mu=0$, and hence
          \begin{equation}\label{EstmCoerL2}
          ((h,Lh))\geq \lambda ((h,h)), \quad \quad \mbox{if} \int h\dd\mu=0
          \end{equation}
           where $\lambda$ can be given by
           \begin{equation}\label{EqLambda}
           \lambda=\lambda_0\min\left\{\frac{1}{2a+1}, \frac{\kappa}{2c\kappa+1}\right\}.
           \end{equation}
          \item Apply Gronwall's lemma and deduce exponential decay in the new inner product,
          \[
          ((e^{-tL}h,e^{-tL}h))\leq e^{-2\lambda t}((h,h)), \quad \quad \mbox{if} \int h\dd\mu=0
          \]
          which, due to the equivalence of the two inner products, implies exponential decay in $H^1(\mu)$-norm
          \[
          ||e^{-tL}h-\int h\dd \mu||_{H^1(\mu)}\leq \frac{c_2}{c_1}e^{-\lambda t}||h-\int h\dd\mu||_{H^1(\mu)}
          \]
         and so the theorem follows by taking $C_0=c_2/c_1$.
        \end{enumerate}

In the coercivity estimate \eqref{EstmCoerL2}, a vital technical point is the introduction of the mixed term $\langle \nabla_x h,\nabla_v h\rangle$. And one has to bound the terms involving  $\nabla^2_x V$ since it appears naturally in the computations. To see this, recall the following expression taken from \cite{Villani},

\begin{eqnarray}\label{CoerDecomp}
  ((h,Lh)) &=& ||\nabla_v h||^2 + a(||\nabla_v^2h||^2 + ||\nabla_vh||^2+\langle \nabla_vh,\nabla_x h\rangle)  \notag\\
  & & +b(2 \langle \nabla^2_vh,\nabla^2_{xv}h\rangle + \langle \nabla_vh,\nabla_x h\rangle+||\nabla_x h||^2 - \langle \nabla_vh,  \nabla_x^2V\cdot\nabla_v h\rangle) \notag\\
  & & +c(||\nabla^2_{xv}h||^2- \langle\nabla_xh,\nabla^2_xV\cdot\nabla_vh\rangle).
\end{eqnarray}

It is then clear that, without the mixed term $\langle \nabla_x h,\nabla_v h\rangle$ (i.e. let $b=0$), there would be no dissipation in the $\nabla_x$ direction, and so it would be impossible to get a coercivity estimate. That way, the inner products $((\cdot,\cdot))$ and $\langle\cdot,\cdot\rangle_{H^1(\mu)}$, though being equivalent, are quite different in coercivity. And we see that the mixed term really helps to get coercivity.

As the computation \eqref{CoerDecomp} shows, in order to obtain a coercivity estimate in the form of \eqref{EstmCoerL1} or \eqref{EstmCoerL2}, we need to bound the terms involving $\nabla_x^2V(x)\cdot\nabla_vh$ which occur in $((h,Lh))$, namely,  $-\langle \nabla_vh,\nabla_x^2 V(x)\cdot \nabla_v h\rangle$ and $-\langle \nabla_xh,\nabla_x^2 V(x)\cdot \nabla_v h\rangle$, in terms of the $L^2$-norm of $\nabla_vh$, $\nabla_v^2 h$, $\nabla_x h$, and $\nabla_{xv}^2h$. And it then becomes natural to consider boundedness conditions in the form of (\ref{HypV2Eq}).

Moreover, assuming the condition \eqref{HypV2Eq} holds with a constant $M$, by Cauchy-Schwartz inequality, we have
\[
((h,Lh))\geq \langle Z, TZ\rangle
\]
with the vector $Z=(||\nabla_vh||,||\nabla^2_{v}h||,  ||\nabla_xh||, ||\nabla^2_{xv}h||) \in \RR^4$ and the symmetric $4\times 4$ matrix $T$ given by
\begin{equation}\label{EsCoMatrixT}
T=\begin{pmatrix}
1+a-b\sqrt{M}& 0 & -(a+b+c\sqrt{M} )/2& -b\sqrt{M}/2\\
0 & a          & 0            & -b\\
-( a+b+c\sqrt{M})/2 & 0          & b            & -c\sqrt{M}/2\\
-b\sqrt{M}/2 & -b         & -c\sqrt{M}/2           &c
\end{pmatrix}.
\end{equation}
To ensure the coercivity estimate \eqref{EstmCoerL1}, it suffices to choose $a,b,c$ such that
\begin{equation}
T\geq \text{Diag}(\lambda_0,0,\lambda_0,0)
\end{equation}
as bilinear forms.  In doing so, the constants $a,b,c$ and $\lambda_0$ depend only on $M$ (and so does $C_0$). For instance, assuming that $M\geq 1$, we could take $a=\frac{1}{25M}$, $b=\frac{1}{200M^2}$, $c=\frac{1}{800M^3}$ and $\lambda_0= \frac{1}{440M^2}$. Then, following the outline above, we obtain a rate of convergence $\lambda$ given by \eqref{EqLambda} which depends only on $M$ and the spectral gap constant $\kappa$.

This shows that  we can get rid of the dependence of the number $N$ of particles, if we can verify the boundedness condition \eqref{HypV2Eq} with a constant $M$ independent of $N$.


\section{Bounded interaction assumption}

We compute at first the Hessian of the interaction potential:
\[
  \nabla^2_{x_ix_j}\left(\frac{1}{2N}\sum\limits_{1\leq k,l\leq N}W(x_k-x_l)\right)
   = \left\{
\begin{aligned}
\frac{1}{N}\sum\limits_{k:k\neq i}\nabla^2W(x_i-x_k),\quad \mbox{if } i=j;\\
-\frac{1}{N} \nabla^2W(x_i-x_j),\quad \mbox{if } i\neq j.
\end{aligned}
\right.
\]
Denote it by $H_{ij}$ for $1\leq i,j\leq N$. It is clear that $H_{ii}=-\sum_{j:j\neq i} H_{ij}$. Put
\[
H_W:=(H_{ij})_{1\leq i,j\leq N},
\]
\[
H_U:= \text{Diag}(\nabla^2 U(x_1),\nabla^2 U(x_2),\cdots,\nabla^2 U(x_N)).
\]
Then we get
\begin{equation}\label{EqHessV}
  \nabla^2 V(x)= (\nabla_{x_ix_j}^2 V(x))_{1\leq i,j\leq N} = H_U + H_{W}.
\end{equation}

We begin by giving an upper bound for the operator norm of the matrix $H_W(x)$. For a real number $r$, as usual, we denote its positive part by $r^+$ and its negative part by $r^{-}$.

\begin{lem}\label{LemHessW}If $|\nabla^2W(y)|_{op}\leq K$ for all $y\in\RR^d$, then
\[
|H_W(x)|_{op}\leq K
\]
for all $x\in \RR^{Nd}$. More precisely, it holds
\begin{enumerate}
  \item If $\nabla^2W\leq \lambda_M I_{d}$, then $H_W(x)\leq \lambda_M^{+} I_{Nd}$ ;
  \item If $\nabla^2W\geq \lambda_m I_{d}$ , then $H_W(x)\geq -\lambda_m^{-} I_{Nd}$.
\end{enumerate}
where the inequalities are understood in the sense of quadratic forms.
\end{lem}

\begin{rmq}
  The coefficient in the above lemma is in fact optimal. Consider $d=1$ and $W(y)=\frac{1}{2}y^2$. In this case, set $p=(1,1,\cdots,1)^T\in\RR^N$, and the matrix $NH_W= N I_{Nd} - pp^T= N \Pi_{p^{\perp}}$ where $\Pi_{p^{\perp}}$ denotes the projection onto the subspace which is perpendicular to $p$. Hence  $H_W$  has two eigenvalues, $1$ and $0$. It follows that the operator norm of $H_W$ is $1$.
\end{rmq}

\begin{proof}Here we use the notation $\langle\cdot,\cdot\rangle$ for the scalar product in the Euclidean spaces. Fix $x=(x_1,x_2,\cdots,x_N)\in\RR^{Nd}$. Let $z=(z_1,z_2,\cdots,z_N)$ where $z_i\in \RR^d$ for $1\leq i\leq N$. Since $H_{ii}=-\sum_{j:j\neq i} H_{ij}$ and $H_{ij}=H_{ji}$, we have
\begin{eqnarray*}
\langle z,H_W z \rangle
&= &
\sum\limits_{j\neq i} \langle z_i,H_{ij}(z_j-z_i)\rangle
=  \sum\limits_{i\neq j} \langle z_j,H_{ji}(z_i-z_j)\rangle\\
&= & -\frac{1}{2} \sum\limits_{i\neq j} \langle z_i-z_j,H_{ji}(z_i-z_j)\rangle\\
&=& \frac{1}{2N} \sum\limits_{i\neq j} \langle z_i-z_j,\nabla^2W(x_i-x_j)\cdot(z_i-z_j)\rangle.
\end{eqnarray*}

\begin{enumerate}
  \item Assume $\nabla^2W\leq \lambda_M I_{d}$, then
  \[
  \langle z_i-z_j,\nabla^2W(x_i-x_j)\cdot(z_i-z_j)\rangle\leq \lambda_M|z_i-z_j|^2
  \]
  and therefore
  \begin{eqnarray*}
  \langle z,H_W z \rangle
  &\leq& \frac{\lambda_M}{2N} \sum\limits_{i\neq j} |z_i-z_j|^2
  =  \frac{\lambda_M}{N} \left(N|z|^2 - |\sum\limits_{i}z_i|^2\right) \\
  &\leq&  \lambda_M^{+}|z|^2.
  \end{eqnarray*}
  \item Assume $\nabla^2W\geq \lambda_m I_{d}$,then
  \[
  \langle z_i-z_j,\nabla^2W(x_i-x_j)\cdot(z_i-z_j)\rangle\geq \lambda_m|z_i-z_j|^2
  \]
  and therefore
  \begin{eqnarray*}
  \langle z,H_W z \rangle
  &\geq& \frac{\lambda_m}{2N} \sum\limits_{i\neq j} |z_i-z_j|^2
  =  \frac{\lambda_m}{N} \left(N|z|^2 - |\sum\limits_{i}z_i|^2\right) \\
  &\geq& - \lambda_m^{-}|z|^2.
  \end{eqnarray*}
  \item $|\nabla^2W|_{op}\leq K$ means that $-K I_{d} \leq \nabla^2W\leq KI_{d}$. By parts (1) and (2), this implies that $-KI_{Nd} \leq H_W\leq KI_{Nd}$ as quadratic forms and hence $|H_W|_{op}\leq K$.
\end{enumerate}
\end{proof}

Lemma \ref{LemHessW} allows us to reduce the boundedness condition \eqref{HypV2Eq} to a simpler one,

\begin{lem}\label{LemVU}Suppose that $|\nabla^2W|_{op} \leq K$. Suppose that there exist positive constants $C_1, C_2$ such that for each $i$ and for all $g\in H^1(m)$,
\begin{equation}\label{HypV2Eq*}
  \int |\nabla^2 U(x_i)|_{op}^2  g^2\dd m \leq C_1\int |\nabla_x g|^2\dd m +  C_2\int g^2\dd m.
\end{equation}
Then  the boundedness condition \eqref{HypV2Eq} is satisfied with a constant $M$ given by
\begin{equation}\label{MLemVU}
  M=\max\{2C_1,2C_2+2K^2\}.
\end{equation}
\end{lem}
\begin{proof}
Under the assumptions and using $\nabla^2V=H_U+H_W$, by Lemma \ref{LemHessW}, we have
\begin{eqnarray*}
\int |\nabla_x^2 V\cdot \nabla_v h|^2\dd\mu
&\leq& 2\int \left(|H_U\cdot \nabla_{v} h|^2 + |H_W\cdot\nabla_{v} h|^2\right)\dd\mu\\
&\leq& 2\int \sum\limits_{1\leq i\leq N}|\nabla^2 U(x_i)|_{op}^2 |\nabla_{v_i} h|^2\dd\mu + 2K^2 \int |\nabla_{v} h|^2\dd\mu.
\end{eqnarray*}

We estimate these terms separately. Apply the inequality \eqref{HypV2Eq*} with $g=\partial_{v_{il}} h$ (here $v_{il}$ is the $l$-th variable of $v_i\in \RR^d$) for $1\leq i\leq N$ and $1\leq l\leq d$, we get
\[
\int |\nabla^2 U(x_i)|_{op}^2 |\partial_{v_{il}} h|^2\dd\mu \leq \int \left[ C_1  \int |\nabla_{x}\partial_{v_{il}}h|^2\dd m(x) + C_2\int|\partial_{v_{il}}h|^2\dd m(x) \right]\dd\gamma(v)
\]
Summing over $i$ and $l$, we have
\[
\int\sum\limits_{1\leq i\leq N} |\nabla^2 U(x_i)|_{op}^2  |\nabla_{v_i}h|^2\dd\mu \leq C_1\int |\nabla^2_{xv}h|_{HS}^2\dd\mu + C_2 \int |\nabla_{v}h|^2\dd\mu.
\]
and so
\[
\int |\nabla_x^2 V\cdot \nabla_v h|^2\dd\mu \leq 2C_1\int |\nabla^2_{xv}h|_{HS}^2\dd\mu + (2C_2+2K^2) \int |\nabla_{v}h|^2\dd\mu.
\]
i.e. the boundedness condition \eqref{HypV2Eq} is satisfied with the constant $M$ given in \eqref{MLemVU}.
\end{proof}

\section{Proof of Theorem \ref{Thm1}}
Let $\H$ be the elliptic generator associated to the mean field measure $m$, that is,
\begin{eqnarray*}
  \H &=& \Delta_x - \nabla V(x)\cdot \nabla_x \\
     &=& \Delta_x - \sum\limits_{1\leq i\leq N}\left(\nabla U(x_i) + \frac{1}{N}\sum\limits_{1\leq j\leq N}\nabla W(x_i-x_j)\right)\cdot\nabla_{x_i}\\
     &=& \sum\limits_{1\leq i\leq N} \H_i
\end{eqnarray*}
where
\[
\H_i=\Delta_{x_i} - \nabla U(x_i)\cdot \nabla_{x_i} - \frac{1}{N}\sum\limits_{1\leq j\leq N}\nabla W(x_i-x_j)\cdot\nabla_{x_i}.
\]

The following known lemma is a key to the Lyapunov type conditions, it was initially proved in \cite{BBCG} to get a Poincar\'e inequality. We include its simple proof for completeness.
\begin{lem}\label{LemLy2}Let $\H$ and $m$ be defined as above, then for all twice-differentiable function $S>0$  and for all $g\in H^1(m)$,
\begin{equation}
\int -\frac{\H S}{S}g^2\dd m \leq \int |\nabla g|^2\dd m.
\end{equation}
\end{lem}
\begin{proof}Indeed, an integration by parts gives
\begin{eqnarray*}
\int -\frac{\H S}{S}g^2\dd m
   &\leq& \int \langle \nabla S, \nabla \frac{g^2}{S}\rangle \dd m(x) \\
   &\leq& \int \langle \nabla S, \frac{2g\nabla g}{S}- \frac{g^2\nabla S}{S^2}\rangle \dd m(x) \\
   &\leq& \int |\nabla g|^2\dd m
\end{eqnarray*}
where the last inequality follows from
\[
\langle 2g\nabla g, \frac{\nabla S}{S}\rangle\leq \frac{g^2|\nabla S|^2}{S^2}+ |\nabla g|^2.
\]
\end{proof}

This second lemma is the heart of the proof. It uses Lyapunov conditions, yet well know for being highly dimensional, but at the marginal level, thus providing results independent of the number of particles.

\begin{lem}\label{PropLy2}
Suppose that the Lyapunov condition \eqref{Thm1EqLy} holds, i.e. there exists positive constants $K_1,K_2$ such that
\[
|\nabla^2 U|_{op}\leq K_1 |\nabla U|+ K_2.
\]
Then for all $g\in H^1(m)$,
\[
\int |\nabla^2 U(x_i)|_{op}^2  g^2\dd m \leq C_1\int |\nabla_x g|^2\dd m +  C_2\int g^2\dd m
\]
with $C_1, C_2$ given by
\begin{equation}\label{PropLy2C}
C_1= 50K_1^2  ,\quad C_2= 4K_2^2 + \frac{25K_1^4d^2}{4}+ \frac{25K'^2K_1^2}{2}.
\end{equation}
\end{lem}

\begin{proof}
Step 1: We show that the Lyaunov condition $|\nabla^2 U|_{op}\leq K_1 |\nabla U|+ K_2$ implies
\begin{equation}\label{CondLya2}
|\nabla^2U|_{op}^2 \leq \eta_1((1-\alpha)|\nabla U|^2-\Delta U)+ \eta_2.
\end{equation}
where
\[
\eta_1=5K_1^2,\eta_2= 4K_2^2+ \frac{25K_1^4d^2}{4}, \mbox{ and } \alpha=\frac{1}{5}.
\]

Indeed, note that
\[
C\Delta U \leq Cd|\nabla^2U|_{op}\leq \epsilon |\nabla^2 U|_{op}^2 + \frac{C^2d^2}{4\epsilon}
\]
for $\epsilon>0$ and $C>0$. And the condition $|\nabla^2 U|_{op}\leq K_1 |\nabla U|+ K_2$ implies
\[
|\nabla^2 U|^2_{op}\leq 2K_1^2 |\nabla U|^2+ 2K_2^2
\]
Then we have
\begin{eqnarray*}
  |\nabla^2U|_{op}^2 + C \Delta U
   &\leq & (1+ \epsilon) |\nabla^2 U|_{op}^2 + \frac{C^2d^2}{4\epsilon} \\
   & =& 2(1+ \epsilon) K^2_1|\nabla U|^2 + 2(1+ \epsilon)K_2^2 + \frac{C^2d^2}{4\epsilon}
\end{eqnarray*}
or
\begin{equation}
|\nabla^2U|_{op}^2 \leq C\left[\frac{2(1+\epsilon)K_1^2}{C}|\nabla U|^2 -\Delta U\right] + 2(1+ \epsilon)K_2^2 + \frac{C^2d^2}{4\epsilon}
\end{equation}
The desired inequality \eqref{CondLya2} follows by taking $\epsilon=1, C=5K_1^2$.

Step 2. We take $S(x)=e^{\alpha U(x_i)/2}$ and compute
\[
\frac{\H S}{S}=\frac{\H_i S}{S}= \frac{\alpha}{2}\left( \Delta U(x_i) + (\frac{\alpha}{2}-1) |\nabla U|^2(x_i) -\frac{1}{N}\sum\limits_{j}\nabla W(x_i-x_j) \cdot \nabla U(x_i)\right)
\]
Since $|\nabla W|\leq K'$, we have
\begin{eqnarray*}
 -\frac{1}{N}\sum\limits_{j}\nabla W(x_i-x_j) \cdot \nabla U(x_i)  &\leq&  K' |\nabla U|(x_i)\\
   &\leq&  \frac{K'^2}{2\alpha} +  \frac{\alpha}{2}|\nabla U|^2(x_i)
\end{eqnarray*}
and so
\begin{eqnarray*}
 \frac{2\H S}{\alpha S} \leq \Delta U(x_i) + (\alpha-1) |\nabla U|^2(x_i)+ \frac{K'^2}{2\alpha} 
\end{eqnarray*}
or
\[
(1-\alpha)|\nabla U|^2(x_i)-\Delta U(x_i) \leq -\frac{2\H S}{\alpha S}+ \frac{K'^2}{2\alpha}
\]
Therefore, by the inequality obtained in Step 1,
\[
|\nabla^2 U(x_i)|_{op}^2 \leq  \eta_1(-\frac{2\H S}{\alpha S}+ \frac{K'^2}{2\alpha})+ \eta_2
\]
Integrating with respect to $g^2\dd m$, we obtain
\begin{eqnarray*}
  \int |\nabla^2 U(x_i)|_{op}^2  g^2\dd m  &\leq & \frac{2\eta_1}{\alpha} \int -\frac{\H S}{S}g^2\dd m + (\eta_2+ \frac{K'^2\eta_1}{2\alpha}) \int g^2\dd m\\
   &\leq& \frac{2\eta_1}{\alpha} \int|\nabla g|^2\dd m + (\eta_2+ \frac{K'^2\eta_1}{2\alpha}) \int g^2\dd m
\end{eqnarray*}
where the last inequality follows from Lemma \ref{LemLy2}.
\end{proof}

\begin{proof}[Proof of Theorem \ref{Thm1}]
By the Lyapunov condition \eqref{Thm1EqLy} in the assumptions, we can apply Lemma \ref{PropLy2} and obtain that for any $g\in H^1(m)$, it holds
\[
\int |\nabla^2 U(x_i)|_{op}^2  g^2\dd m \leq C_1\int |\nabla_x g|^2\dd m +  C_2\int g^2\dd m
\]
with $C_1, C_2$ given by \eqref{PropLy2C} for instance which are independent of the number $N$ of particles.

Next, using Lemma \ref{LemVU}, the boundedness condition \eqref{HypV2Eq} holds with $M$ given by
\[
 M=\max\{2C_1,2C_2+2K^2\}.
\]
We apply Villani's Hypocoercivity theorem \ref{Thm-VillaniH1} and then obtain the result.
\end{proof}

\section{Proof of Theorem \ref{Thm2}}
The next results extend the ones in the previous section to unbounded $\nabla W$. Instead, we shall require that the mean field measure $m$ satisfies the Uniform Logarithmic Sobolev Inequality. We prove the following estimate first, relying only on the variational formulation of entropy.

\begin{lem}\label{lemLy3LSI}Assume that the measure $m$ satisfies a log-Sobolev inequality with a constant $C_{LS}$. For $0< \tau < \frac{1}{4 C_{LS}}$ given and for each $i$ fixed, it holds for all suitably integrable function $g$ that
\begin{equation}
\int \frac{1}{N-1}\sum\limits_{j:j\neq i}|x_i-x_j|^2 g^2 \dd m \leq
\frac{2C_{LS}}{\tau} \int |\nabla g|^2\dd m + \frac{d\ln(1-4\tau C_{LS})^{-1}}{2\tau}\int g^2\dd m.
\end{equation}
In particular, taking $\tau = \frac{1}{8 C_{LS}}$, it holds
\begin{equation}
\int \frac{1}{N-1}\sum\limits_{j:j\neq i}|x_i-x_j|^2 g^2 \dd m \leq
16C_{LS}^2 \int |\nabla g|^2\dd m + 4\ln{2} \cdot dC_{LS}\int g^2\dd m.
\end{equation}
\end{lem}

\begin{proof}Put \[F(x)= \frac{1}{N-1}\sum\limits_{j:j\neq i}|x_i-x_j|^2 \]
Since the measure $ m$ satisfies a log-Sobolev inequality, we can apply the classical entropy inequality
\[
\int f g^2 \dd m \leq \ent_{m}(g^2) + \int g^2\dd m \log\int e^{f}\dd m
\]
with $f=\tau F$. Then, for any $\tau>0$ such that $c_2=\log \int e^{\tau F}\dd m $ is finite, we obtain
      \begin{eqnarray*}
       \int F  g^2\dd m&\leq&  \frac{1}{\tau} \ent_{m}(g^2) + \frac{1}{\tau} \int g^2\dd m \log \int e^{\tau F } \dd m \\
      &\leq & \frac{2C_{LS}}{\tau}  \int |\nabla_{x}g|^2\dd m + \frac{c_2}{\tau} \int g^2\dd m
      \end{eqnarray*}
where the last inequality follows from the log Sobolev inequality for $m$.

Now it remains to give an upper bound of $\int e^{\tau F}\dd m$. Thanks to the symmetry of $m(\dd x_1,\dd x_2,\cdots,\dd x_N)$, we find
\begin{eqnarray*}
  \int e^{\tau F}\dd m &\leq& \int \frac{1}{N-1}\sum\limits_{j:j\neq i}e^{\tau|x_i-x_j|^2}\dd m(x) \\
  &=& \int e^{\tau |x_1-x_2|^2}\dd m(x)
\end{eqnarray*}
Let $\dd\gamma_1(y)= (2\pi)^{-d/2}e^{-|y|^2/2}\dd y$ be the standard gaussian measure on $\RR^d$. Due to the identity $e^{\tau |x|^2}= \int e^{\sqrt{2\tau}x\cdot y}\dd\gamma_1(y)$, we have
\begin{eqnarray*}
  \int e^{\tau |x_1-x_2|^2}\dd m(x)
  &=& \int \int e^{\sqrt{2\tau}(x_1-x_2)\cdot y} \dd\gamma_1(y)\dd m(x)\\
   &=& \int \dd\gamma_1(y) \int e^{\sqrt{2\tau}(x_1-x_2)\cdot y}\dd m(x)
\end{eqnarray*}
For any given $y\in \RR^d$, the function $\sqrt{2\tau}(x_1-x_2)\cdot y$ has mean zero w.r.t the measure $m$. Indeed this is a consequence of symmetry,
\[
\int (x_1-x_2)\cdot y \dd m(x)=\int x_1\cdot y \dd m(x)-\int x_2\cdot y \dd m(x)=0.
\]
And note that $\sqrt{2\tau}(x_1-x_2)\cdot y$ is a Lipschitz function of $x$  with Lipschitz constant $2\sqrt{\tau}|y|$. Therefore, according to the exponential integrability under a logarithmic Sobolev inequality (see \cite[Chapter 5]{BGL} for instance), the function $\sqrt{2\tau}(x_1-x_2)\cdot y$ satisfies
\[
\int e^{\sqrt{2\tau}(x_1-x_2)\cdot y}\dd m(x) \leq e^{2\tau|y|^2C_{LS}}
\]
for any $y\in \RR^d$. Hence, if $0<\tau< 1/(4C_{LS})$, we obtain
\begin{eqnarray*}
  \int e^{\tau F}\dd m &\leq& \int e^{2\tau C_{LS} |y|^2}\dd\gamma_1(y) \\
  &= & (1-4\tau C_{LS})^{-d/2}
\end{eqnarray*}
and then the desired estimate follows.
\end{proof}

\begin{lem}\label{PropLy4}Suppose that the mean field measure $m$ satisfies a log-Sobolev inequality with a constant $C_{LS}$. Suppose the Lyapunov condition \eqref{Thm1EqLy}
and
\[
|\nabla^2 W|_{op}\leq K.
\]
Then, for all $g\in H^1(m)$,
\[
 \int |\nabla^2 U(x_i)|_{op}^2  g^2\dd m  \leq C_1\int|\nabla g|^2\dd m + C_2 \int g^2\dd m
\]
with the constants $C_1, C_2$ given by
\begin{equation}\label{PropLy4C}
C_1= 50K_1^2 (1+ 4K^2C_{LS}^2) ,\quad C_2= 4K_2^2+\frac{25K_1^4d^2}{4}+ 50\ln 2 \cdot dK^2K_1^2C_{LS}.
\end{equation}
\end{lem}

\begin{proof}
As in the proof of Lemma \ref{PropLy2}, the Lyapunov condition \eqref{Thm1EqLy} implies
\begin{equation}\label{Cond.Ly1}
|\nabla^2U|_{op}^2 \leq \eta_1((1-\alpha)|\nabla U|^2-\Delta U)+ \eta_2
\end{equation}
with $\eta_1=5K_1^2,\eta_2= 4K_2^2+ \frac{25K_1^4d^2}{4}$ and $\alpha=\frac{1}{5}$.

Consider $S(x)=e^{\alpha U(x_i)/2}$ and compute
\[
\frac{\H S}{S}= \frac{\alpha}{2}\left( \Delta U(x_i) + (\frac{\alpha}{2}-1) |\nabla U|^2(x_i) -\frac{1}{N}\sum\limits_{j:j\neq i}\nabla W(x_i-x_j) \cdot \nabla U(x_i)\right)
\]
By Cauchy-Schwartz inequality, it holds
\begin{eqnarray*}
 -\frac{1}{N}\sum\limits_{1\leq j\leq N}\nabla W(x_i-x_j) \cdot \nabla U(x_i)
   &\leq&  \frac{1}{2\alpha}|\frac{1}{N}\sum\limits_{1\leq j\leq N}\nabla W(x_i-x_j)|^2+ \frac{\alpha}{2}|\nabla U|^2(x_i)\\
   &\leq & \frac{1}{2\alpha N}\sum\limits_{1\leq j\leq N}|\nabla W(x_i-x_j)|^2+ \frac{\alpha}{2} |\nabla U|^2(x_i)
\end{eqnarray*}
and so
\begin{eqnarray*}
 \frac{2\H S}{\alpha S} &\leq& \Delta U(x_i) + (\alpha-1) |\nabla U|^2(x_i)+ \frac{1}{2\alpha N}\sum\limits_{1\leq j\leq N}|\nabla W(x_i-x_j)|^2
\end{eqnarray*}
Using the assumption on $\nabla^2U$, we have
\[
|\nabla^2 U(x_i)|_{op}^2\leq \eta_1\left(-\frac{2\H S}{\alpha S} + \frac{1}{2\alpha N}\sum\limits_{1\leq j\leq N}|\nabla W(x_i-x_j)|^2\right)+ \eta_2
\]
Integrating with respect to $g^2\dd m$, we obtain by lemma \ref{LemLy2}
\begin{eqnarray*}
  \int |\nabla^2 U(x_i)|_{op}^2  g^2\dd m  &\leq & \frac{2\eta_1}{\alpha} \int -\frac{\H S}{S}g^2\dd m + \eta_2 \int g^2\dd m +\frac{\eta_1}{2\alpha}\Theta \\
  &\leq & \frac{2\eta_1}{\alpha}\int|\nabla g|^2\dd m +  \eta_2 \int g^2\dd m +\frac{\eta_1}{2\alpha}\Theta
\end{eqnarray*}
with
\[
\Theta := \frac{1}{  N}\int\sum\limits_{1\leq j\leq N}|\nabla W(x_i-x_j)|^2g^2\dd m.
\]

To prove the lemma, it remains to show that
\begin{equation}\label{IneqLyW}
   \Theta \leq 16K^2C_{LS}^2 \int |\nabla g|^2\dd m + 4\ln{2} \cdot dK^2C_{LS}\int g^2\dd m
\end{equation}
Since $W$ is even, we see that $\nabla W(0)=0$. Then it follows from the assumption $|\nabla^2W|_{op}\leq K$ that
\[
|\nabla W(z)|\leq |\nabla W(0)|+ K|z|\leq K|z|
\]
therefore
\[
\Theta \leq \int \frac{1}{N-1}\sum\limits_{j:j\neq i}|\nabla W(x_i-x_j) |^2 g^2\dd m \leq K^2\int \frac{1}{N-1}\sum\limits_{j:j\neq i}|x_i-x_j|^2 g^2\dd m
\]
So we can apply the lemma \ref{lemLy3LSI} to get the inequality \eqref{IneqLyW} and the proof is then complete.
\end{proof}

Now we turn to the

\begin{proof}[Proof of Theorem \ref{Thm2}]
By the Lyapunov condition \eqref{Thm1EqLy} in the assumptions, we can apply Lemma \ref{PropLy4} and obtain that for any $g\in H^1(m)$, it holds
\[
\int |\nabla^2 U(x_i)|_{op}^2  g^2\dd m \leq C_1\int |\nabla_x g|^2\dd m +  C_2\int g^2\dd m
\]
with $C_1, C_2$ given by \eqref{PropLy4C}. Note that these constants are independent of the number $N$ of particles.

Next, owing to Lemma \ref{LemVU}, we know the boundedness condition \eqref{HypV2Eq} holds with
\[
 M=\max\{2C_1,2C_2+2K^2\}
\]
We apply Villani's Hypocoercivity theorem \ref{Thm-VillaniH1} and then obtain the convergence with rates independent of the number $N$ of particles.
\end{proof}

\section{An improvement on the rate of convergence}
The boundedness conditions proved in the previous sections share the following form
\[
\int |\nabla_x^2 V\cdot \nabla_v h|^2\dd\mu \leq M_1\int |\nabla^2_{xv}h|_{HS}^2\dd\mu + M_2 \int |\nabla_{v}h|^2\dd\mu.
\]
where the coefficients $M_1$ and $M_2$ might be
\[
M_1=2C_1,\quad M_2=2C_2+2K^2
\]
with constants $C_1$ and $C_2$ being given in \eqref{PropLy2C} or \eqref{PropLy4C}. Note that $C_1$ and $C_2$ depend on $K_1$ and $K_2$ in the Lyapunov condition \eqref{Thm1EqLy}
\[
 |\nabla^2U|_{op} \leq K_1|\nabla U|+ K_2.
\]
It is clear that $K_1$ is related to the asymptotic behaviour of $\nabla^2 U$ and $\nabla U$ at infinity, while $K_2$ is more relevant to the local properties. For instance, when $U$ behaves as a polynomial at infinity, $K_1$ can be taken to be arbitrarily close to zero (with the price of $K_2$ being large); consequently, $M_1$ might be very small while $M_2$ might be large. This suggests that in general we can obtain a boundedness condition with very different $M_1$ and $M_2$.

In this section, we shall take advantage of this fact and get a slight improvement on the rate of convergence $\lambda$. As mentioned before, the rate of convergence in \cite[Theorem 35]{Villani} is of order $M^{-2}$, as $M\rightarrow \infty$ with $M=\max\{1,M_1,M_2\}$. However, by distinguishing the two constants $M_1$ and $M_2$, the rate can be improved to be of order $M_2^{-1/2}$ for small $M_1$ and big $M_2$.

\begin{prop}
If the following boundedness condition holds,
\[
\int |\nabla_x^2 V\cdot \nabla_v h|^2\dd\mu \leq M_1\int |\nabla^2_{xv}h|_{HS}^2\dd\mu + M_2 \int |\nabla_{v}h|^2\dd\mu,
\]
then the rate of convergence $\lambda$ can be taken to be of order $\frac{1}{\sqrt{M_2}}$ for small $M_1$ and big $M_2$.
\end{prop}

\begin{rmq}We consider mainly the behaviour of $\lambda$ when $M_2$ is large while $M_1$ is small. For specific $M_1$ and $M_2$, an refinement of the method is always needed to get a better rate of convergence.
\end{rmq}

\begin{proof}
We set in this proof that $M=\max\{1,M_2\}$. By Cauchy-Schwartz inequality and the boundedness condition above,
\begin{eqnarray*}
  - \langle \nabla_vh,  \nabla_x^2V\cdot\nabla_v h\rangle&\geq& -||\nabla_vh|| ||\nabla_x^2V\cdot\nabla_v h||  \\
   &\geq&  -||\nabla_vh|| \sqrt{M_1||\nabla^2_{xv}h|^2 + M_2 ||\nabla_{v}h||^2} \\
   &\geq& -||\nabla_vh|| (\sqrt{M_1}||\nabla^2_{xv}h| + \sqrt{M_2} ||\nabla_{v}h||)
\end{eqnarray*}
Similarly,
\[
- \langle\nabla_xh,\nabla^2_xV\cdot\nabla_vh\rangle \geq -||\nabla_xh|| (\sqrt{M_1}||\nabla^2_{xv}h| + \sqrt{M_2} ||\nabla_{v}h||)
\]
This leads to $$((h,Lh))\geq \langle Z, T'Z\rangle $$ with a matrix $T'$ given by
\[
T'=\begin{pmatrix}
1+a-b\sqrt{M_2}& 0 & -(a+b+c\sqrt{M_2} )/2& -b\sqrt{M_2}/2\\
0 & a          & 0            & -b\\
-( a+b+c\sqrt{M_2})/2 & 0          & b            & -c\sqrt{M_1}/2\\
-b\sqrt{M_1}/2 & -b         & -c\sqrt{M_1}/2           &c
\end{pmatrix}.
\]
Denote
\[
S=(S_{ij})_{1\leq i,j\leq 4}:= T'- \text{Diag}(\lambda_0,0,\lambda_0,0),
\]
\[
Z=(Z_1,Z_2,Z_3,Z_4).
\]
The object is then choose $a,b,c$ such that $S$ is positive definite. If now it is assumed that
\begin{equation}\label{Eq7-1}
  b\sqrt{M_2}\leq \frac14,\quad \lambda_0=\frac{b}{4}\leq \frac14,
\end{equation}
then
\begin{equation}\label{Eq7-2}
  S_{11}=1+a-b\sqrt{M_2}- \lambda_0\geq a+ \frac12, \quad S_{33}=b-\lambda_0= \frac34 b.
\end{equation}

And if we impose furthermore the conditions below
\begin{equation}\label{Eq7-3}
\frac12 \cdot \frac{b}{2}\geq \left(\frac{a+b+c\sqrt{M_2}}{2}\right)^2,\quad a\cdot \frac{c}{8}\geq \left(\frac{b\sqrt{M_1}}{2}\right)^2, \quad a\cdot \frac{c}{2}\geq b^2,\quad \frac{b}{4}\cdot \frac{3c}{8}\geq \left(\frac{c\sqrt{M_1}}{2}\right)^2,
\end{equation}
then we have
\[
\frac12 Z_1^2 + \frac{b}{2} Z_3^2\geq |2S_{13}Z_1Z_3|,\quad aZ_1^2 +  \frac{c}{8}Z_4^2\geq|2S_{14}Z_1Z_4|,
\]
\[
a Z_2^2 + \frac{c}{2} Z_4^2\geq |2S_{24}Z_2Z_4|,\quad \frac{b}{4}Z_3^2 +  \frac{3c}{8}Z_4^2\geq |2S_{34}Z_3Z_4|,
\]
and it follows that
\begin{eqnarray*}
  \langle Z, SZ\rangle &=& S_{11}Z_1^2 + S_{22}Z_2^2 + S_{33}Z_3^2 + S_{44}Z_4^2 +2S_{13}Z_1Z_3+ 2S_{14}Z_1Z_4 + 2S_{24}Z_2Z_4 + 2S_{34}Z_3Z_4\\
   &\geq& S_{11}Z_1^2 + S_{22}Z_2^2 + S_{33}Z_3^2 + S_{44}Z_4^2 - (\frac12 Z_1^2 + \frac{b}{2} Z_3^2) -(aZ_1^2 +  \frac{c}{8}Z_4^2)\\
   & & -(a Z_2^2 + \frac{c}{2} Z_4^2)-(\frac{b}{4}Z_3^2 +  \frac{3c}{8}Z_4^2)\\
   &=& (S_{11}-\frac12-a)Z_1^2  + (S_{33}-\frac{3b}{4})Z_3^2 \\
   &\geq& 0
\end{eqnarray*}
where the last inequality follows from \eqref{Eq7-2}.

\textbf{Case 1:} To fix ideas, we consider the case $M_1\leq1$ first. In this case, we may take $M_1$ as $1$, then the conditions \eqref{Eq7-3} become
\begin{equation}\label{Eq7-3a}
b\geq \left(a+b+c\sqrt{M_2}\right)^2,\quad ac\geq 2b^2,\quad \frac{3}{8}b\geq c.
\end{equation}
For the moment let $\alpha,\beta,\gamma$ be the constants such that
\[
a=\alpha /\sqrt[4]{M}, \quad b= \beta/ \sqrt[4]{M}^2,  c= \gamma/ \sqrt[4]{M}^3.
\]
then, since  $M=\max\{1,M_2\}\geq 1$, it suffices that
\begin{equation}\label{Eq7-3b}
\beta\geq (\alpha + \beta + \gamma)^2,\quad \alpha \gamma \geq 2\beta^2, \quad \frac{3}{8}\beta \geq \gamma
\end{equation}
where $\beta\leq 1/4$ (so that $b\sqrt{M_2}\leq 1/4$). To conclude we may take all these inequalities to be equalities, and in this case
\[
\beta = \frac{576}{25921},\quad \alpha = \frac{16}{3}\beta = \frac{3072}{25921},\quad \gamma= \frac{3}{8}\beta=\frac{216}{25921}.
\]
and then
\[
\lambda_0= \frac{b}{4}= \frac{144}{25921\sqrt{M}}.
\]
Recall the equality \eqref{EqLambda} says that
$$\lambda=\lambda_0\min\{\frac{1}{2a+1}, \frac{\kappa}{2c\kappa+1}\}= \frac{144}{25921\sqrt{M}}\min\{\frac{1}{\frac{6144}{25921\sqrt[4]{M}}+1}, \frac{\kappa}{\frac{512\kappa}{25921\sqrt[4]{M}^3}+1}\}.$$

In particular, note that $M=\max\{1,M_2\}=M_2$ for $M_2\geq 1$, hence the rate of convergence $\lambda$ is of order $1/\sqrt{M_2}$ for large $M_2$.

\textbf{Case 2:} Now we consider the case $M_1>1$. The conditions \eqref{Eq7-3} become
\begin{equation}\label{Eq7-3a}
b\geq \left(a+b+c\sqrt{M_2}\right)^2,\quad ac\geq 2M_1b^2,\quad \frac{3}{8}b\geq M_1c.
\end{equation}
The solution to the corresponding system of equalities is given by
\[
b=\frac{1}{\left(\frac{16}{3}M_1^2 +1 + \frac{3\sqrt{M_2}}{8M_1}\right)^2},\quad a = \frac{16}{3}M_1^2b,\quad c= \frac{3}{8M_1} b.
\]
which gives a rate of convergence of order $M_2^{-1}$ for large $M_2$. Or we can proceed as in Case $1$, and we may take
\[
b=\frac{1}{\left(\frac{16}{3}M_1^2 +1 + \frac{3}{8M_1}\right)^2\sqrt{M}},\quad a = \frac{16}{3\sqrt[4]{M}}M_1^2,\quad c= \frac{3}{8M_1\sqrt[4]{M}^3} .
\]
which gives a rate of convergence of order $1/\sqrt{M_2}$ for large $M_2$.
\end{proof}

\begin{rmq}
To ensure the positiveness of the matrix $T'$ in the proof, the constant $\lambda_0$ must satisfy
\[
1+a-b\sqrt{M_2}-\lambda_0\geq 0, \mbox{ and } b-\lambda_0\geq 0.
\]
Assume $a\leq 1$, then the first inequality implies that $b\leq 2/\sqrt{M_2} $ while the second one implies $\lambda_0\leq b$. As a consequence, $\lambda\leq\lambda_0$ is at most of order $1/\sqrt{M_2}$. The rate of convergence stated in the proposition is sharp in this sense.

Consider the matrix $T$ given in \eqref{EsCoMatrixT} in section 3, similarly the rate of convergence $\lambda$ is at most of order $1/\sqrt{M}$. Furthermore, a fine argument shows that the positiveness of $T$ requires that $\lambda_0$ is at most of order $M^{-2}$ as $M$ tends to infinity. So the distinction between $M_1$ and $M_2$ allows us to get a better growth control for $\lambda$ (for large $M_2$).
\end{rmq}

{\bf Acknowledgements: } A. Guillin and C. Zhang are supported by Project EFI ANR-17-CE40-0030 of the French National Research Agency. W. Liu is supported by the NSFC 11731009. These results were first presented at the conference SMAI 2019 in Guidel Plages. C. Zhang are grateful for the organizers for this magnificent event and especially for the mini-symposium ``In\'egalit\'es fonctionnelles en probabilit\'es et analyse, et applications".

\bibliographystyle{plain}

\end{document}